\documentclass[reqno]{amsart}
\usepackage{amssymb,xypic,graphicx,hyperref,url}
\usepackage{thmtools,thm-restate,cleveref}

\newtheorem{theorem}{Theorem}
\numberwithin{theorem}{section}
\newtheorem{proposition}[theorem]{Proposition}

\newtheorem{lemma}[theorem]{Lemma}
\newtheorem{conjecture}[theorem]{Conjecture}
\newtheorem{question}[theorem]{Question}
\theoremstyle{definition}
\newtheorem{definition}[theorem]{Definition}
\newtheorem{remark}[theorem]{Remark}
\newtheorem{example}[theorem]{Example}

\newcommand{\ccl}{\bar}
\newcommand{\0}{\emptyset}
\newcommand{\fld}{\Bbbk}
\newcommand{\st}{~|~}
\newcommand{\sm}{\setminus}
\newcommand{\HH}{\tilde{H}}
\newcommand{\defeq}{\mathrel{\mathop:}=}

\newcommand{\Zz}{\mathbb{Z}}
\DeclareMathOperator{\link}{link}
\DeclareMathOperator{\depth}{depth}
\DeclareMathOperator{\sdepth}{sdepth}

\begin{document}
\title{A non-partitionable Cohen--Macaulay simplicial complex}

\author[A.\ M.\ Duval]{Art M.\ Duval}
\address{Department of Mathematical Sciences\\ University of Texas at El Paso}
\email{aduval@utep.edu}
\author[B.\ Goeckner]{Bennet Goeckner}
\address{Department of Mathematics \\ University of Kansas}
\email{bennet@ku.edu}
\author[C.\ J.\ Klivans]{Caroline J.\ Klivans}
\address{Division of Applied Mathematics and Department of Computer Science \\ Brown University}
\email{klivans@brown.edu}
\author[J.\ L.\ Martin]{Jeremy L.\ Martin}
\address{Department of Mathematics \\ University of Kansas}
\email{jlmartin@ku.edu}
\thanks{This work was partially supported by a grant from the Simons Foundation (grant number 315347 to J.L.M.)}
\date{\today}


\begin{abstract}
A long-standing conjecture of Stanley states that every Cohen--Macaulay simplicial complex is partitionable.  We disprove the conjecture by constructing an explicit counterexample.  Due to a result of Herzog, Jahan and Yassemi, our construction also disproves the conjecture that the Stanley depth of a monomial ideal is always at least its depth.
\end{abstract}
\maketitle

\section{Introduction}
Cohen--Macaulay simplicial complexes are ubiquitous in algebraic and topological combinatorics.  They 
were introduced in 1975 by Stanley in his celebrated proof of the Upper Bound Conjecture for spheres~\cite{Stanley_UBC}.\footnote{See~\cite{Stanley_howUBC} for his engaging and personal account of how the proof came to be.}
 The theory of Cohen--Macaulay \emph{rings} has long been of great importance in algebra and algebraic geometry; see, e.g., \cite{Rees,ZariskiSamuel,Grothendieck,Hochster,Hochster-ICM,BH}.  The connection to combinatorics via what is known as Stanley--Reisner theory was established by Hochster \cite{Hochster}, Reisner \cite{Reisner}, and Stanley~\cite{St5}; standard references for this subject are \cite{GreenBook} and \cite{BH}.

The focus of this article is the following conjecture, described by Stanley as ``a central combinatorial conjecture on Cohen--Macaulay complexes'' \cite[p.~85]{GreenBook}.
It was originally proposed by Stanley~\cite[p.~149]{Stanley} in 1979 and independently by Garsia~\cite[Remark 5.2]{Garsia} in 1980 for order complexes of Cohen--Macaulay posets.

\begin{conjecture}[Partitionability Conjecture]\label{main-conj}
Every Cohen--Macaulay simplicial complex is partitionable.
\end{conjecture}

We explicitly construct a Cohen--Macaulay complex that is not partitionable, thus disproving the Partitionability Conjecture.   In fact, we give a general method for constructing counterexamples and an explicit infinite family of non-partitionable Cohen--Macaulay complexes.  We begin by giving some background for the conjecture, which will also be directly relevant in our construction.

Two basic invariants of a simplicial complex $\Delta$ are its $f$- and $h$-vectors
\[f(\Delta)=(f_{-1}(\Delta),f_0(\Delta),\dots,f_d(\Delta)), \qquad h(\Delta)=(h_{0}(\Delta),h_1(\Delta),\dots,h_{d+1}(\Delta)),\]
where $d=\dim\Delta$.  The number $f_i = f_i(\Delta)$ is the number of $i$-dimensional faces (simplices) in~$\Delta$.   The $h$-vector is more subtle.  It carries the same information as the $f$-vector (the two are related by an invertible linear transformation), and arises naturally in algebra: the Hilbert series of the Stanley--Reisner ring of~$\Delta$ is $(1-t)^{-d}\sum_jh_j(\Delta)t^j$.  (See Section~\ref{prelim} for precise definitions.)  It is not at all apparent if the numbers $h_j(\Delta)$ have a combinatorial interpretation; for instance, they need not be positive in general.

A \emph{partitioning} of a pure simplicial complex $\Delta$ is a decomposition into pairwise-disjoint Boolean intervals whose maximal elements are exactly the facets (maximal faces) of~$\Delta$.  Partitionability was introduced by Provan~\cite{Provan} and Ball~\cite{Ball} in the context of reliability analysis.  For a partitionable complex, the $h$-numbers enumerate the minimum elements of the intervals by size.  In particular, shellable complexes are easily seen to be partitionable, and hence their $h$-vectors have this interpretation. 
The strict inclusions
\[\{\text{shellable complexes}\}\subsetneq\{\text{constructible complexes}\}\subsetneq\{\text{Cohen--Macaulay complexes}\}\]
are also well known.  For example, the nonshellable balls constructed by Rudin~\cite{Rudin} and Ziegler~\cite{Ziegler} are constructible (see also \cite{Lutz0}), and any triangulation of the dunce hat is Cohen--Macaulay but not constructible \cite[\S2]{Hachimori-2d}.
 On the other hand, the possible $h$-vectors of Cohen--Macaulay, constructible, and shellable complexes are all the same \cite[Theorem~6]{Stanley-CMC}, suggesting that their entries ought to count something explicit.
The Partitionability Conjecture
 would have provided a combinatorial interpretation of the $h$-vectors of Cohen--Macaulay complexes.

The idea of our construction is to work with \emph{relative} simplicial complexes.
Suppose $Q=(X,A)$ is a relative simplicial complex that is not partitionable, but with $X$ and $A$ Cohen--Macaulay.
Theorem~\ref{glue} gives a general method of gluing together sufficiently many copies of $X$ along $A$ to obtain a counterexample to the Partitionability Conjecture, provided that $A$ is an {induced} subcomplex of~$X$.
This reduces the problem to finding an appropriate pair~$(X,A)$.
Our starting point is the nonshellable simplicial 3-ball $Z$ constructed by Ziegler~\cite{Ziegler}, in which we find a suitable subcomplex~$A$ and in turn the desired relative complex~$Q$ (Theorem~\ref{relativenotpart}).   By refining the construction, we are able to obtain, in Theorem~\ref{smaller}, a Cohen--Macaulay non-partitionable complex that is much smaller than predicted by Theorem~\ref{glue}, with $f$-vector 
$(1, 16, 71, 98, 42)$ and $h$-vector $(1, 12, 29)$.

The existence of a Cohen--Macaulay nonpartitionable complex has an important consequence in commutative algebra.  For a polynomial ring $S=\fld[x_1,\dots,x_n]$ and a $\Zz^n$-graded $S$-module $M$, many fundamental algebraic invariants of $M$, such as its dimension and multigraded Hilbert series, can be profitably studied using combinatorics.  On the other hand, the combinatorial properties of the \emph{depth} of $M$ are less well understood.  In \cite{Stanley-LDE}, Stanley proposed a purely combinatorial analogue of depth, defined in terms of certain vector space decompositions of $M$.  This invariant, now known as the \emph{Stanley depth} and written $\sdepth M$, has attracted considerable recent attention (see \cite{WhatIs} for an accessible introduction to the subject, and \cite{Herzog-survey} for a comprehensive survey), centering around the following conjecture of Stanley \cite[Conjecture~5.1]{Stanley-LDE}:

\begin{restatable}[Depth Conjecture]{conjecture}{depthconj}
\label{depthconj}
For all $\Zz^n$-graded $S$-modules $M$,
\[\sdepth M \geq \depth M.\]
\end{restatable}

Herzog, Jahan, and Yassemi proved \cite[Corollary~4.5]{HJY} that when $I$ is the Stanley--Reisner ideal of a Cohen--Macaulay complex $\Delta$, the inequality $\sdepth S/I\geq\depth S/I$ is equivalent to the partitionability of $\Delta$.  Therefore, our counterexample to the Partitionability Conjecture disproves the Depth Conjecture as well.  We exhibit a smaller counterexample to the Depth Conjecture using a relative complex in Remark~\ref{smaller-relative};  see Section~\ref{sdepth-subsection}.

It was also previously not known whether all constructible complexes were
partitionable; see, e.g., \cite[\S4]{Hachimori-CCR}. The counterexample we obtain is not only
Cohen--Macaulay, but in fact constructible.  Therefore, even constructibility does not imply partitionability.

\section{Preliminaries}
\label{prelim}

\subsection{Simplicial and relative simplicial complexes}  Throughout the paper, all complexes will be finite.  Let $V$ be a finite set.  A \emph{simplicial complex} on $V$ is a collection $\Delta$ of subsets of $V$ such that
whenever $\sigma \in \Delta$ and $\tau \subseteq \sigma$, then $\tau \in \Delta$.  Equivalently, $\Delta$ is an order ideal in the Boolean poset $2^V$.  The symbol $|\Delta|$ denotes the standard geometric realization of $\Delta$.  The elements of $\Delta$ are called the \emph{faces} of~$\Delta$, and the elements of~$V$ are \emph{vertices}.  Maximal faces are called \emph{facets}.  The \emph{dimension} of a face $\sigma$ is $\dim\sigma=|\sigma|-1$, and the dimension of~$\Delta$ is $\dim\Delta=\max\{\dim\sigma\st\sigma\in\Delta\}$.  We often write $\Delta^d$ to indicate that $\dim\Delta=d$.
A complex is \emph{pure} if all maximal faces have the same dimension.  
A \emph{subcomplex} of $\Delta$ is a simplicial complex $\Gamma$ with $\Gamma\subseteq\Delta$.  A subcomplex is an \emph{induced subcomplex} if it is of the form
\[\Delta|_W\defeq\{\sigma\in\Delta\st\sigma\subseteq W\}\]
for some $W\subseteq V$. 

In the construction of our counterexample, we will work with the more general class of \emph{relative simplicial complexes}.  A relative complex $\Phi$ on $V$ is a subset of $2^V$ that is \emph{convex}: if $\rho,\tau\in\Phi$ and $\rho\subseteq\sigma\subseteq\tau$, then $\sigma\in\Phi$.  We sometimes refer to simplicial complexes as ``absolute'' to distinguish them from relative complexes.

Every relative complex can be expressed as a pair $\Phi=(\Delta,\Gamma):=\Delta\sm\Gamma$, where $\Delta$ is a simplicial complex and $\Gamma\subseteq\Delta$ is a subcomplex.  Topologically, $\Phi$ corresponds to the quotient space $|\Delta|/|\Gamma|$.
Note that there are infinitely many possibilities for the pair $\Delta,\Gamma$.  The unique minimal expression is obtained by letting $\Delta=\ccl{\Phi}$ be the \emph{combinatorial closure} of $\Phi$, i.e., the smallest simplicial complex containing $\Phi$ as a subset, and setting $\Gamma=\Delta\sm\Phi$.  Note that in this case $\dim\Gamma<\dim\Delta$, because the maximal faces of $\Delta$ are precisely those of $\Phi$.

The notation $\HH_i(\Delta)$ denotes the $i^{th}$ reduced simplicial homology group with coefficients in~$\Zz$.  (The underlying ring does not matter for our purposes.)
The simplicial homology groups $\HH_i(\Phi)$ of a relative complex $\Phi=(\Delta,\Gamma)$ are just the relative homology groups $\HH_i(\Delta,\Gamma)$ in the usual topological sense (see, e.g., \cite{Hatcher}); in particular, the homology groups of $\Delta$, $\Gamma$, and $\Phi$ fit into a long exact sequence.

The \emph{$f$-vector} of an (absolute or relative) complex $\Delta^d$ is $f(\Delta)=(f_{-1}, f_0, \ldots, f_d)$, where $f_i = f_i(\Delta)$ is the number of $i$-dimensional faces of~$\Delta$.  Note that $f_{-1}(\Delta) = 1$ for every absolute complex other than the void complex $\Delta=\0$.
The \emph{$h$-vector} $h(\Delta) = (h_0, h_1, \dots, h_{d+1})$ is defined by
\[h_k=\sum_{i=0}^k (-1)^{k-i} \binom{d+1-i}{k-i} f_{i-1}, \qquad 0\leq k\leq d+1.\]
In particular, the $f$- and $h$-vectors determine each other.

The \emph{link} of a face $\sigma\in\Delta$ is defined as
\[\link_\Delta (\sigma) \defeq  \{ \tau \in \Delta\st \tau \cap \sigma = \0,\ \tau \cup \sigma \in \Delta \}.\]
Observe that if $\Delta^d$ is pure and $\dim\sigma=k$, then $\dim\link_\Delta(\sigma)=d-k-1$.  If $\sigma$ is a facet of $\Delta$ then $\link_\Delta(\sigma)=\{\0\}$, the trivial complex with only the empty face, and if $\sigma\not\in\Delta$ then we set $\link_\Delta(\sigma)$ to be the void complex with no faces.

If $\Phi = (\Delta, \Gamma)$ is a relative complex and $\sigma \in \Delta$, we can define the relative link by
\[
\link_{\Phi}(\sigma)
~=~ (\link_\Delta(\sigma),\link_\Gamma(\sigma)).
\]
It is easy to check that this construction is intrinsic to $\Phi$, 
i.e., it does not depend on the choice of the pair $\Delta,\Gamma$.  Note that $\link_\Phi(\sigma)$ is not necessarily a subset of~$\Phi$.  

\subsection{Cohen--Macaulay simplicial complexes}
A ring is \emph{Cohen--Macaulay} if its depth equals its (Krull) dimension.  Reisner's criterion \cite[Theorem~1]{Reisner} states that Cohen--Macaulayness of the Stanley--Reisner ring \cite[\S II.1]{GreenBook} of a simplicial complex can be expressed in terms of simplicial homology, and we will take this criterion as our definition.  The relative version of Reisner's criterion is Theorem~5.3 of~\cite{Stanley_relative}.
\begin{theorem} \cite{Reisner,Stanley_relative}\label{combinatorics-defn}
A simplicial complex $\Delta$ is \emph{Cohen--Macaulay} if for every face $\sigma \in \Delta$, 
\begin{equation} \label{reisner}
\HH_i(\link_\Delta(\sigma))=0 \quad\text{ for } i<\dim\link_\Delta(\sigma).
\end{equation}
Similarly, a relative complex $\Phi = (\Delta, \Gamma)$ is \emph{Cohen--Macaulay} if for every $\sigma \in \Delta$, 
$$\HH_i(\link_\Delta(\sigma),\link_\Gamma(\sigma))=0 \quad\text{ for } i<\dim\link_\Delta(\sigma).$$
\end{theorem}
In fact, Cohen--Macaulayness is a \emph{topological} invariant: it depends only on the homeomorphism type of the geometric realization $|\Delta|$. This was proved by Munkres \cite{Munkres}.  Topological invariance holds for relative complexes as well~\cite[Corollary~III.7.3]{GreenBook}.  Importantly, if $|\Delta|$ is homeomorphic to a ball or to a sphere, then $\Delta$ is Cohen--Macaulay \cite[\S2]{Munkres}.

The following technical lemma will be central to our construction.

\begin{lemma} \label{gluing-lemma}
Let $\Delta_1$ and $\Delta_2$ be $d$-dimensional Cohen--Macaulay simplicial complexes on disjoint vertex sets.  
Let $\Gamma$ be a Cohen--Macaulay simplicial complex of dimension~$d$ or~$d-1$, and suppose that each $\Delta_i$ contains a copy of $\Gamma$ as an induced subcomplex.  Then the complex $\Omega$ obtained by identifying the two copies of $\Gamma$ (or ``gluing together $\Delta_1$ and $\Delta_2$ along $\Gamma$'') is Cohen--Macaulay.
\end{lemma}

\begin{proof}
It is clear that $\Omega$ is a CW-complex.  The requirement that each copy of $\Gamma$ is an \emph{induced} subcomplex of $\Delta_i$ means that $\Omega$ is in fact a simplicial complex (because faces with the same underlying vertex set will be identified).  It remains to show that $\Omega$ is Cohen--Macaulay.  Henceforth, to simplify the notation, we will identify $\Gamma$ with $\Delta_1\cap\Delta_2$, so that $\Omega$ is identified with $\Delta_1\cup\Delta_2$.

Let $\sigma$ be a face of $\Omega$.  Note that
\begin{equation} \label{link-union-int}
\begin{aligned}
\link_\Omega (\sigma) &=  \{ \tau \in \Omega\st \tau \cap \sigma = \0,\ \tau \cup \sigma \in \Omega \} = \link_{\Delta_1} (\sigma) \cup \link_{\Delta_2} (\sigma),\\
\link_\Gamma (\sigma) &=  \{ \tau \in \Gamma\st \tau \cap \sigma = \0,\ \tau \cup \sigma \in \Gamma \} = \link_{\Delta_1} (\sigma) \cap \link_{\Delta_2} (\sigma).
\end{aligned}
\end{equation}
First, suppose that $\sigma \in \Delta_1\sm\Delta_2$. Then Reisner's criterion~\eqref{reisner} holds for $\sigma$ because $\link_\Omega(\sigma) = \link_{\Delta_1} (\sigma)$, and ${\Delta_1}$ is Cohen--Macaulay.  Likewise, Reisner's criterion holds for faces of $\Delta_2\sm\Delta_1$.

On the other hand, suppose that $\sigma\in\Gamma$. 
Then the observations~\eqref{link-union-int} give rise to a reduced Mayer-Vietoris sequence
\[
\cdots\to\HH_i(\link_{\Delta_1} (\sigma)) \oplus  \HH_i(\link_{\Delta_2} (\sigma)) 
\to  \HH_i(\link_\Omega(\sigma)) 
\to \HH_{i-1}(\link_\Gamma(\sigma))\to\cdots.
\]
But since $\Delta_1,\Delta_2,\Gamma$ are Cohen--Macaulay and
$\dim\Gamma\geq\dim\Delta_1-1$, the Mayer-Vietoris sequence implies that $\HH_i(\link_\Omega(\sigma))=0$ for all $i<d-\dim\sigma-1$.  This is precisely the statement that Reisner's criterion holds for $\sigma$.
\end{proof}

Iterating Lemma~\ref{gluing-lemma}, we obtain immediately:
\begin{proposition} \label{gluing-prop}
Let $\Delta_1, \ldots, \Delta_n$ be $d$-dimensional Cohen--Macaulay simplicial complexes on disjoint vertex sets.  
Let $\Gamma$ be a Cohen--Macaulay simplicial complex of dimension~$d-1$ or~$d$, and suppose that each~$\Delta_i$ contains a copy of~$\Gamma$ as an induced subcomplex.  Then the complex~$\Omega$ obtained from $\Delta_1,\dots,\Delta_n$ by identifying the $n$ copies of~$\Gamma$ is Cohen--Macaulay.
\end{proposition}

\subsection{Shellability, partitionability, and constructibility}
\begin{definition} \label{define-shelling}
Let $\Delta$ be a pure simplicial complex.  A \emph{shelling} of~$\Delta$ is a total ordering $F_1,\dots,F_n$ of its facets so that for every $j$, the set
\[\{\sigma\subseteq F_j \st \sigma\not\subseteq F_i \text{ for all } i<j\}\]
has a unique minimal element $R_j$. 
\end{definition}

The $h$-vector of a shellable complex has a simple combinatorial interpretation:
\begin{equation} \label{interpret-h}
h_k(\Delta) = \#\{j \st \#R_j=k\}.
\end{equation}
In particular $h_k(\Delta)\geq0$ for all $k$, and in fact $h_k(\Delta)=0$ implies $h_\ell(\Delta)=0$ for all $\ell>k$ (a consequence of \cite[Theorem~5.1.15]{BH}).
Shellable complexes are Cohen--Macaulay, although the converse is not true: well-known counterexamples include any triangulation of the dunce hat, as well as the nonshellable balls constructed by Rudin \cite{Rudin} and Ziegler \cite{Ziegler}.  On the other hand, Cohen--Macaulay complexes satisfy the same conditions on the $h$-vector, so it is natural to look for a combinatorial interpretation of their $h$-vectors.

\begin{definition} \label{define-partitioning}
Let $\Delta$ be a pure simplicial complex with facets $F_1,\dots,F_n$.  A \emph{partitioning} $\mathcal{P}$ of $\Delta$ is a decomposition into pairwise-disjoint Boolean intervals
\begin{equation*}
\Delta=\bigsqcup_{i=1}^n\, [R_i,F_i]
\end{equation*}
where $[R_i,F_i]=\{\sigma\in\Delta\st R_i\subseteq\sigma\subseteq F_i\}$.  We say that each $F_i$ is \emph{matched} to the corresponding $R_i$.
\end{definition}

Clearly, shellable complexes are partitionable.
If $\Delta$ is partitionable, then its $h$-vector automatically carries the combinatorial interpretation \eqref{interpret-h} \cite[Proposition~III.2.3]{GreenBook}.  Moreover, Definitions~\ref{define-shelling} and~\ref{define-partitioning}, and the interpretation of the $h$-vector, carry over precisely from absolute to relative complexes.

\begin{example}
A partitionable complex need not be Cohen--Macaulay, much less shellable.  The following example is due to Bj\"orner \cite[p.~85]{GreenBook}.  Let $\Delta$ be the pure 2-dimensional complex with facets $123,124,134,234,156$ (abbreviating $\{1,2,3\}$ by 123, etc.).  This complex is not Cohen--Macaulay because vertex 1 fails Reisner's criterion~\eqref{reisner}, but it is partitionable:
\[\Delta=[\0,156]\cup[2,123]\cup[3,134]\cup[4,124]\cup[234,234].\]
In particular, $h(\Delta)=(1,3,0,1)$, which is not the $h$-vector of any Cohen--Macaulay complex
(since $h_2=0$ and $h_3>0$).
\end{example}

Constructibility, introduced by Hochster \cite{Hochster}, is a combinatorial condition intermediate between shellability and Cohen--Macaulayness.

\begin{definition} \label{define-constructible}
  A complex $\Delta^d$ is \emph{constructible} if it is a simplex, or if it can be written as $\Delta=\Delta_1\cup\Delta_2$, where $\Delta_1$, $\Delta_2$, and $\Delta_1\cap\Delta_2$ are constructible of dimensions $d$, $d$, and $d-1$ respectively.
\end{definition}

Hachimori~\cite{Hachimori-CCR} investigated the question of whether constructibility implies partitionability. Our counterexample to the Partitionability Conjecture is in fact constructible, resolving this question as well.

\section{The counterexample}\label{sec-counterexample}

We first give a general construction that reduces the problem of finding a counterexample to the problem of constructing a certain kind of non-partitionable Cohen--Macaulay relative complex.

\begin{theorem}\label{glue}
Let $Q=(X,A)$ be a relative complex such that 
\begin{enumerate}
\item[(i)] $X$ and $A$ are Cohen--Macaulay;
\item[(ii)] $A$ is an induced subcomplex of $X$ of codimension at most~1; and 
\item[(iii)] $Q$ is not partitionable.
\end{enumerate}
Let $k$ be the total number of faces of $A$, let $N>k$, and let $C=C_N$ be the simplicial complex constructed from $N$ disjoint copies of $X$ identified along the subcomplex $A$.
Then $C$ is Cohen--Macaulay and not partitionable.  
\end{theorem}

\begin{proof}
First, $C$ is Cohen--Macaulay by Proposition~\ref{gluing-prop}.
Second, suppose that $C$ has a partitioning $\mathcal{P}$.   Let $X_1, X_2, \ldots, X_N$ be the $N$ copies of $X$.   By the pigeonhole principle, since $N>k$, there is some copy of $X$, say $X_N$, none of whose facets is matched to a face in $A$.
 Let $[R_1,F_1],\dots,[R_\ell,F_\ell]$ be the intervals in $\mathcal{P}$ for which $F_i\in X_N$; then
\begin{equation} \label{oops}
\bigcup_{i=1}^\ell [R_i,F_i] \quad\subseteq\quad X_N\sm A.
\end{equation}
No other interval in $\mathcal{P}$ can intersect $X_N\sm A$ nontrivially, so in fact equality must hold in \eqref{oops}.  But then \eqref{oops} is in fact a partitioning of $X_N\sm A=Q$, which was assumed to be non-partitionable.
\end{proof}

\begin{remark} \label{induced-min-vertex}
It is easy to see that a subcomplex $A\subset X$ is an induced subcomplex if and only if every minimal face of $X\sm A$ has dimension~0.  Therefore, this condition may be viewed as a restriction on the relative complex $(X,A)$.
\end{remark}

\subsection{The construction}
Throughout, we abbreviate the simplex on vertices $\{v_1,\dots,v_k\}$ by $v_1\cdots v_k$.
Our construction begins with Ziegler's nonshellable
$3$-ball~$Z$, which is a nonshellable 
triangulation of the $3$-ball with $10$ vertices labeled $0,1,\dots,9$, and the following $21$ facets~\cite[\S4]{Ziegler}:
\[\begin{array} {ccccccc}
0123, & 0125, & 0237, & 0256, & 0267, & 1234, & 1249,\\
1256, & 1269, & 1347, & 1457, & 1458, & 1489, & 1569,\\
1589, & 2348, & 2367, & 2368, & 3478, & 3678, & 4578.
\end{array}\]
The complex $Z$ is not shellable, but it is partitionable, Cohen--Macaulay 
and, in fact, constructible \cite{Hachimori}.

Let $B$ be the induced subcomplex $Z|_{\{0,2,3,4,6,7,8\}}$.
That is, $B$ is the pure 3-dimensional complex with facets
\[\begin{array} {ccccccc}
0237, & 0267, & 2367, & 2368, & 2348, & 3678, & 3478.
\end{array}\]
The given order is a shelling of~$B$; in particular $B$ is Cohen--Macaulay. Define $Q$ to be the relative complex $Q=(Z,B)$.  Then $Q$ is also Cohen--Macaulay by~\cite[Corollary~3.2]{Art1}.

The facets of $Q$ are
\begin{equation} \label{Qfacets}
\begin{array}{ccccccc}
1249, & 1269, & 1569, & 1589, & 1489, & 1458, & 1457, \\
4578, & 1256, & 0125, & 0256, & 0123, & 1234, & 1347.
\end{array}
\end{equation}
The minimal faces of $Q$ are just the vertices 1, 5, 9.
We can picture $Q$ easily by considering its combinatorial closure $\ccl{Q}$, that is, the 3-dimensional simplicial complex generated by the facets \eqref{Qfacets}.  In fact $\ccl{Q}$ is a shellable ball; the ordering of facets given in \eqref{Qfacets} is a shelling.
The complement $A=\ccl{Q}\sm Q=\ccl{Q}|_{\{0,2,3,4,6,7,8\}}$ is the shellable 2-ball with facets
\begin{equation}\label{Afacets}
\begin{array}{ccccc}
026, & 023, & 234, & 347, & 478.
\end{array}
\end{equation}
Thus $Q=(\ccl{Q},A)$.  The $f$- and $h$-vectors of these complexes are
\begin{align*}
f(\bar{Q}) &= (1,10,31,36,14), & h(\bar{Q}) &= (1,6,7,0,0),\\
f(A) &= (1,7,11,5,0),  & h(A) &= (1,4,0,0,0),\\
f(Q) &= (0,3,20,31,14),  & h(Q)&=(0,3,11,0,0).
\end{align*}
The 1-skeleton of $\ccl{Q}$ is shown in Figure~\ref{show-cclQ}.
\begin{figure}
\includegraphics[width=2.75in]{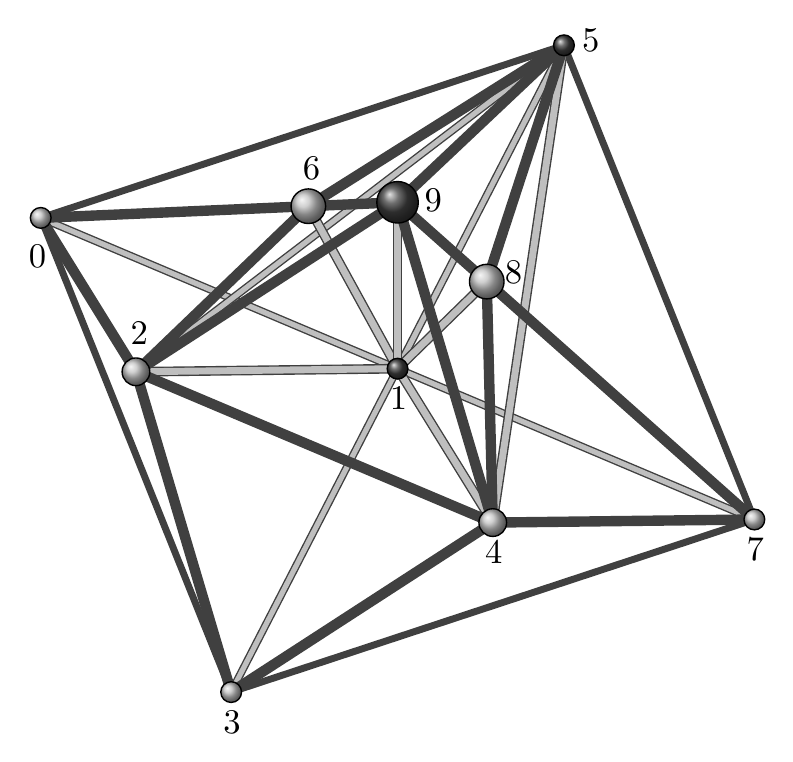}
\caption{A perspective view of the 1-skeleton of $\ccl{Q}$ from the top.  Dark edges are exterior edges visible from the top; light edges are interior edges, or exterior edges at the bottom.  Light vertices are in $A$.} \label{show-cclQ}
\end{figure}
The triangles on the boundary of $Q$, i.e., those contained in exactly one facet, are illustrated in Figure~\ref{showQ}, which shows the boundary of $Q$ as seen
from the front (left) and back (right).  The five shaded triangles are the facets of $A$, and hence are missing from $Q$.  
\begin{figure}
\includegraphics[width=4in]{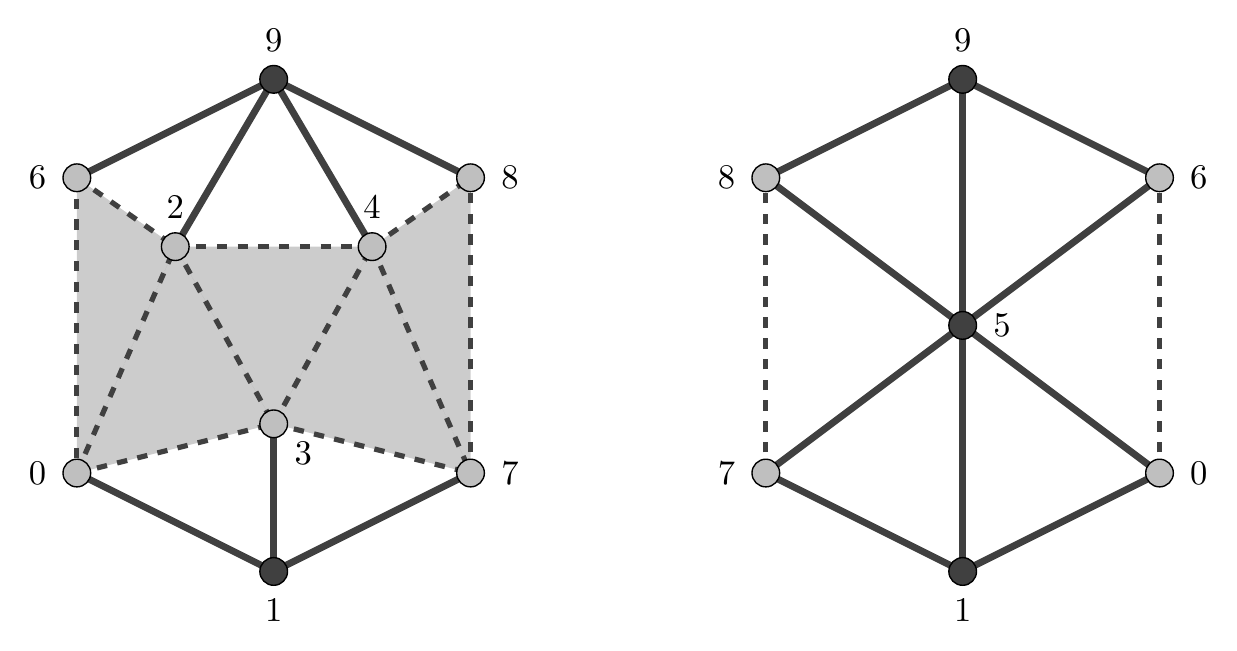}
\caption{\emph{Left:} A front view of $Q$.   \emph{Right:} A back view of~$Q$. The shaded and dashed faces are in $A$.} \label{showQ}
\end{figure}

In what follows, we will use the fact 
that the triple transposition $\tau=(0\;7)(2\;4)(6\;8)$ is a simplicial automorphism of $\bar{Q}$.  This symmetry is apparent as a reflection through the plane containing vertices 1, 3, 5, and 9 in Figure~\ref{show-cclQ}, and as a vertical reflection in each part of Figure~\ref{showQ}.

\begin{theorem}\label{relativenotpart}
The relative complex $Q$ is not partitionable.
\end{theorem}

\begin{proof} 
Suppose that $Q$ admits a partitioning $\mathcal{P}$.  We will show that a particular minimal face, namely vertex 5, must simultaneously belong to two intervals of the partitioning, which is a contradiction.

For each facet $F\in Q$, denote by $I_F=[R_F,F]$ the interval of~$\mathcal{P}$ with top element~$F$. 

For each triangle $T$ on the boundary, there is only one interval that can contain~$T$.  
In particular, $489\in I_{1489}$.  It follows that
$148\not\in I_{1489}$,
for otherwise $148\cap489=48\in I_{1489}$, but $48\not\in Q$.  Therefore
$148\in I_{1458}$,
since $1458$ is the only other facet containing $148$.  Then
$458\not\in I_{1458}$,
again because $148\cap458=48\not\in Q$, and thus $45\notin I_{1458}$.
The other two facets that contain 45 are 4578 and 1457.  Therefore, either $45\in I_{4578}$ or $45\in I_{1457}$.  On the other hand, these are also the only two facets that contain the edge 57.  Since
\begin{equation*} \label{observe:5}
45,57 ~\subset~ 457 ~\subset~ 1457,4578
\end{equation*}
the edges 45 and 57 must belong to the same interval of $\mathcal{P}$ (namely, whichever one of $I_{1457},I_{4578}$ contains 457).
But then that interval must also contain $45\cap57=5$.  We have shown that
\begin{equation} \label{observe:6}
\text{either } 5 \in I_{1457} \text{ or } 5 \in I_{4578}.
\end{equation}

By applying the automorphism $\tau$ to the above argument, we conclude that 
\begin{equation} \label{observe:7}
\text{either } 5 \in I_{0125} \text{ or } 5 \in I_{0256}.
\end{equation}
But \eqref{observe:6} and \eqref{observe:7} cannot both be true, and we have reached a contradiction.\end{proof}

We can now give an explicit description of our counterexample to the Partitionability Conjecture.
Since $X=\ccl{Q}$ and $A$ are both shellable balls, they are Cohen--Macaulay.  We may therefore apply Theorem~\ref{glue}, with $N=25$ (since $A$ has 24 faces total).

\begin{theorem}\label{first-counterexample}
Let $X=\ccl{Q}$ be the combinatorial closure of $Q$, and let $A=X\sm Q$.  That is, $X$ and $A$ are the absolute simplicial complexes whose facets are listed in~\eqref{Qfacets} and~\eqref{Afacets}, respectively.  Then the simplicial complex $C_{25}$ constructed in
Theorem~\ref{glue} is Cohen--Macaulay and non-partitionable.  
\end{theorem}
The $f$-vector is $f(C_{25})=f(A)+25f(Q)=(1,82,511,780,350)$.

For this particular construction, the full power of Theorem~\ref{glue} is not necessary; there is a much smaller counterexample.

\begin{theorem}\label{smaller}
  Let $Q$, $A$, and $X=\ccl{Q}$ be as described above.  Then the simplicial complex $C_3$ obtained by gluing together three copies of $X$ along $A$ is Cohen--Macaulay and non-partitionable.
\end{theorem}
\begin{proof}
Suppose that $C_3$ is partitionable.  By the pigeonhole principle, at least one of the three copies of $Q$ inside $C_3$ has no facets matched to either edge~48 or its image under $\tau$, edge~26.  These two edges are the only two faces of $A$ that occur in the argument of Theorem~\ref{relativenotpart}.  Therefore, that argument applies to this copy of $Q$, and once again we conclude that \eqref{observe:6} and \eqref{observe:7} must both hold, a contradiction.
\end{proof}

The $f$-vector is $f(C_3) = f(A)+3f(Q) = (1,16,71,98,42)$.  We do not know if there exists a smaller counterexample (for example, the complex $C_2$ obtained by gluing two copies of $X$ together along $A$ is partitionable).  In particular, it is still open whether every \emph{two}-dimensional Cohen--Macaulay simplicial complex  is partitionable; see Hachimori \cite{Hachimori-2d}.

We have previously observed that $X$ and $A$ are shellable.  We note that $X$ and $A$ are contractible, and it is easily seen that $X$ deformation-retracts onto~$A$, so $C_3$ is contractible as well, although it is not homeomorphic to a ball.

\begin{remark} \label{smaller-relative}
There is a much smaller \emph{relative} simplicial complex that is Cohen--Macaulay but not partitionable, with $f$-vector $(0,0,5,10,5)$. 
This complex can be written as $Q'=(X',A')$, where $X'=\overline{Q'}=Z|_{\{1,4,5,7,8,9\}}$ is the complex with facets
\[1589, \ 1489, \ 1458, \ 1457, \ 4578,\]
and $A'$ is the (non-induced) subcomplex of $X'$ with facets
\[489, \ 589, \ 578, \ 157.\]
These complexes are shellable balls of dimensions~3 and~2 respectively (the given orders of facets are shelling orders), and $A'$ is contained in the boundary of $X'$ (note that each facet in $A'$ is contained in only one facet of $X'$), so $Q'$ is Cohen--Macaulay by \cite[Corollary~5.4]{Stanley_relative}.
On the other hand, one can check directly that there is no partitioning of $Q'$.
Because $A'$ is not an induced subcomplex, it is not 
possible to obtain a counterexample to the Partitionability Conjecture by applying Theorem~\ref{glue}.
\end{remark}

\begin{remark} \label{constructible-counterexample}
It is easily seen that $C_3$ is constructible.  Therefore, it furnishes a counterexample not only to the Partitionability Conjecture, but also to the conjecture that every constructible simplicial complex is partitionable \cite[\S4]{Hachimori-CCR}.  Furthermore, since all constructible complexes are Cohen--Macaulay \cite[p.~219]{BH}, the constructibility and non-partitionability of $C_3$ are sufficient to disprove the Partitionability Conjecture.
\end{remark}

\subsection{Stanley depth}\label{sdepth-subsection}
Let $\fld$ be a field and $S=\fld[x_1,\dots,x_n]$, and let $M$ be a $\Zz^n$-graded $S$-module.  A \emph{Stanley decomposition} $\mathcal{D}$ of $M$ is a vector space decomposition
\[M = \bigoplus_{i=1}^r \fld[X_i]\cdot m_i\]
where each $X_i$ is a subset of $\{x_1,\dots,x_n\}$ and each $m_i$ is a homogeneous element of $M$.
The \emph{Stanley depth} of $M$ is defined as
\[\sdepth M = \max\limits_{\mathcal{D}} \left\{\min(|X_1|,\dots,|X_r|)\right\},\]
where $\mathcal{D}$ ranges over all Stanley decompositions of~$M$.  If $\Phi$ is an (absolute or relative) simplicial complex, then we define its Stanley depth to be the Stanley depth of its associated Stanley--Reisner ring or module.
This invariant has received substantial recent attention \cite{WhatIs,Herzog-survey}, centering on the Depth Conjecture of Stanley \cite[Conjecture~5.1]{Stanley-LDE}, which we now restate.
\depthconj* 

Herzog, Jahan and Yassemi \cite[Corollary~4.5]{HJY} proved that if $\Delta$ is a Cohen--Macaulay simplicial complex whose Stanley--Reisner ring \cite[\S II.1]{GreenBook} is $\fld[\Delta] := S/I_{\Delta}$ (so that $\depth\fld[\Delta]=\dim\fld[\Delta]=\dim\Delta+1$), then Conjecture~\ref{depthconj} holds for $\fld[\Delta]$ if and only if $\Delta$ is partitionable.  Therefore, our construction provides a counterexample to the Depth Conjecture.  Katth\"{a}n has conjectured that the inequality $\sdepth S/I\geq\depth S/I-1$ holds for every monomial ideal~$I$; for a detailed exposition and the evidence for this conjecture, see \cite{Katthan-Betti}.

A smaller counterexample to Conjecture~\ref{depthconj} is provided by the relative complex $Q'$ in Remark~\ref{smaller-relative}. 
The depth of each of $C_3$ and $Q'$ is easily seen to be $4$, but the Stanley depth of each of $C_3$ and $Q'$ is $3$.  The Stanley depth computations were made by Katth\"{a}n \cite{Lukas}, using the algorithm developed by Ichim and Zarojanu \cite{IchZar}.

\section{Open questions}\label{CQ}

Now that we know that Cohen--Macaulayness and even constructibility are not sufficient to guarantee partitionability, it is natural to ask
what other conditions do suffice.  Hachimori defined a related but more restricted class of \emph{strongly constructible} complexes and showed that they are always partitionable~\cite[Corollary~4.7]{Hachimori-CCR}.  Here are two additional possibilities, inspired by what our counterexample $C_3$ is \emph{not}.  First, $C_3$ is not homeomorphic to a ball, because the triangles in $A$ are each contained in three facets.  On the other hand, balls are Cohen--Macaulay, motivating the following question:

\begin{question}
Is every simplicial ball partitionable?
\end{question}

This conjecture is true if we further assume the ball is convexly realizable, by \cite[Proposition~III.2.8]{GreenBook}; see also \cite{KS}.  On the other hand, there exist non-convex simplicial balls in dimensions as small as~3; see, e.g., \cite{Lutz1,Lutz2}.

Garsia \cite[Remark~5.2]{Garsia} proposed the Partitionability Conjecture for the special class of order complexes of \emph{Cohen--Macaulay posets} (see also \cite{Bac1,Bac2,BGS}), which give rise to balanced Cohen--Macaulay simplicial complexes.
Recall that a $d$-dimensional simplicial complex is \emph{balanced} if its vertices can be colored with $d+1$ colors so that every facet has one vertex of each color.  For instance, if $P$ is a ranked poset, then its order complex is easily seen to be balanced by associating colors with ranks.  The complex $\ccl{Q}$ with facets listed in \eqref{Qfacets} is not balanced (because its 1-skeleton is not 4-colorable), hence neither is $C_3$ or $C_{25}$, nor indeed $C_N$ for any~$N$.

\begin{question}
Is every balanced Cohen--Macaulay simplicial complex partitionable?
\end{question}

Although Cohen--Macaulay complexes are not necessarily partitionable,
their $h$-vectors are still nice; they are always non-negative and in fact coincide
with the $h$-vectors of shellable complexes.  Without the Partitionability Conjecture, the question remains:

\begin{question}
What does the $h$-vector of a Cohen--Macaulay simplicial complex count?
\end{question}

One answer is given by~\cite{DuvalZhang}, where it is shown that every
simplicial complex can be decomposed into \emph{Boolean trees}
indexed by iterated Betti numbers; see \cite[Corollary~3.5]{DuvalZhang}.  The starting point of that paper is a conjecture of Kalai \cite[Conjecture~22]{Kalai} that any simplicial complex can
be partitioned into intervals in a way related to algebraic shifting.
Kalai's conjecture would have implied that simplicial complexes could be
decomposed into Boolean intervals.
Such a decomposition into intervals, however, would have implied the Partitionability Conjecture.  Hence
our result provides a counterexample to Kalai's conjecture.
Moreover, the decomposition in~\cite{DuvalZhang} may be best possible at this level of generality.  

\section*{Acknowledgements}
We are grateful to Pedro Felzenszwalb for valuable discussions and
generous help with coding and Matlab implementations.  
We thank Margaret Bayer, Louis Billera, Gil Kalai, and Christos Athanasiadis
for helpful discussions and suggestions.
The open-source computer
algebra system Sage \cite{Sage} and Masahiro Hachimori's online
library of simplicial complexes \cite{Hachimori} were valuable
resources.

\bibliographystyle{alpha}
\bibliography{biblio-fullnames}
\end{document}